\documentclass[
	,12pt%
	,pagesize%
	,headings=small%
	,paper=a4%
	,parskip=false%
	,abstract=true
	,toc=bibliography
	,DIV=12%
]{scrartcl}

\newcommand{\showbrief}{}

\addtokomafont{sectioning}{\normalfont\bfseries}
\setkomafont{title}{\normalfont}
\setkomafont{subtitle}{\normalfont}
\setkomafont{author}{\normalfont}
\setkomafont{date}{\normalfont}
\addtokomafont{pageheadfoot}{\scshape\small}
\setkomafont{caption}{\footnotesize}
\setkomafont{captionlabel}{\usekomafont{caption}\bfseries}
\setcapindent{0pt}

\usepackage{xspace}
\usepackage{bm}
\usepackage[utf8]{inputenc}
\usepackage[T1]{fontenc}
\usepackage{csquotes}
\usepackage[english]{babel}
\usepackage{graphicx}
\graphicspath{{Pictures/}{/}}
\usepackage[export]{adjustbox}	% key=value interface

\usepackage[dvipsnames]{xcolor}
\usepackage{tikz-cd}
\usetikzlibrary{backgrounds}

\usepackage[draft]{fixme}
\usepackage[nointlimits]{amsmath}
\usepackage{amssymb,mathrsfs,mathtools}
\usepackage{newtxtext}
\usepackage[varvw]{newtxmath}
\AtBeginDocument{%
  \mathchardef\standardeq=\mathcode`=
  \mathchardef\standardless=\mathcode`<
  \mathchardef\standardgreater=\mathcode`>
  \mathcode`="8000
  \mathcode`<"8000
  \mathcode`>"8000
}
\begingroup\lccode`~`=\lowercase{\endgroup
  \def~}{\mathrel{\mspace{0.936mu}\standardeq\mspace{0.936mu}}}
\begingroup\lccode`~`<\lowercase{\endgroup
  \def~}{\mathrel{\mspace{0.153mu}\standardless\mspace{0.153mu}}}
\begingroup\lccode`~`>\lowercase{\endgroup
  \def~}{\mathrel{\mspace{0.153mu}\standardgreater\mspace{0.153mu}}}
\usepackage{upgreek}
\usepackage{braket}
\usepackage{cancel}
\usepackage{siunitx}
\usepackage{aliascnt} %alias counter
\usepackage[amsmath,thmmarks]{ntheorem}
\usepackage[expansion=true,protrusion=true]{microtype}
\usepackage{xkeyval}
\usepackage{authblk}

\usepackage[final,%
	pdftex,%
	bookmarks,%
	bookmarksdepth=3,%
	breaklinks=true,%
	colorlinks=true,%
	urlcolor=NavyBlue,%
	linkcolor=NavyBlue,%
	citecolor=ForestGreen,%	
]{hyperref}%
\usepackage[all]{hypcap}

\AtBeginEnvironment{abstract}{\footnotesize}
\makeatletter
\patchcmd{\@maketitle}{\huge}{\Large}{}{}
\makeatother

\newcommand{\Curve}{\gamma}

\DeclareDocumentCommand{\Path}{ o }{
	\IfValueTF{#1}{%
		\varGamma_{\!#1}%
	}{%
		\varGamma%
	}%
}

\DeclareDocumentCommand{\dotPath}{ o }{
	\IfValueTF{#1}{%
		\dot{\varGamma}_{\!#1}%
	}{%
		{\dot{\varGamma}}%
	}%
}

\newcommand{\Hmeasure}{\mathcal{H}}

\DeclareDocumentCommand{\dist}{ o o o }{
	\IfValueTF{#1}{
		\varrho_{#1}\IfValueTF{#2}{(#2,#3)}{}
	}{
		\varrho\IfValueTF{#2}{(#2,#3)}{}
	}
}

\DeclareDocumentCommand{\distC}{ o o }{
	\varrho_{\gamma}\IfValueTF{#1}{(#1,#2)}{}
}

\DeclareDocumentCommand{\LineEl}{ o o }{
	\IfValueTF{#1}{
	  \omega_{#1}\IfValueTF{#2}{(#2)}{}
	}{
	  \omega\IfValueTF{#2}{(#2)}{}
	}
}

\DeclareDocumentCommand{\dLineEl}{ o o }{
	\IfValueTF{#1}{
	  \dd \omega_{#1}\IfValueTF{#2}{(#2)}{}
	}{
    \dd \omega\IfValueTF{#2}{(#2)}{}
	}
}

\DeclareDocumentCommand{\LineElC}{ o }{
	\IfValueTF{#1}{
        \omega_\Curve(#1)
	}{
        \omega_\Curve
	}
}

\DeclareDocumentCommand{\dLineElC}{ o }{
	\IfValueTF{#1}{
	  \dd \omega_\Curve(#1)
	}{
        \dd \omega_\Curve
	}
}

\DeclareDocumentCommand{\LebesgueM}{ o }{
	\IfValueTF{#1}{
        \lambda(#1)
	}{
        \lambda
	}
}

\DeclareDocumentCommand{\dLebesgueM}{ o }{
	\IfValueTF{#1}{
	  \dd \lambda(#1)
	}{
        \dd \lambda
	}
}

\DeclareDocumentCommand{\Speed}{ o }{
	\IfValueTF{#1}{
	  	h_{#1}
	}{
      	h
	}
}

\DeclareDocumentCommand{\InvSpeed}{ o }{
	\IfValueTF{#1}{
	  	H_{#1}
	}{
      	H
	}
}

\newcommand{\NotSure}[1]{{\color{Cyan}{#1}}}

\newcommand{\brief}[1]{\ifthenelse{\isundefined{\showbrief}}{}{{\color{NavyBlue}{\bigskip\emph{Brief:} #1\newline}}}}%

\newcommand{\Circle}{\mathbb{T}}

\newcommand{\trace}{\mathrm{tr}}

\newcommand{\AmbDim}{m} % dimension of ambient space
\newcommand{\AmbSpace}{{\R^\AmbDim}}

\DeclareDocumentCommand{\Hess}{ O{} }{\operatorname{Hess}_{#1}}

\DeclareDocumentCommand{\converges}{ o }{
	\mathbin{%
		\IfValueTF{#1}{%
			\mathrel{\vbox{\offinterlineskip\ialign{%
				\hfil##\hfil\cr
				$\scriptscriptstyle#1$\cr
				$-\!\!\!-\!\!\!\rightarrow$\cr
			}}}
		}{%
			-\!\!\!-\!\!\!\rightarrow
		}%
	}%
}

\DeclareDocumentCommand{\wconverges}{ o }{
	\mathbin{%
		\IfValueTF{#1}{%
			\mathrel{\vbox{\offinterlineskip\ialign{%
				\hfil##\hfil\cr
				$\scriptscriptstyle#1$\cr
				$-\!\!\!-\!\!\!\rightharpoonup$\cr
			}}}
		}{%
			-\!\!\!-\!\!\!\rightharpoonup
		}%
	}%
}

\newcommand{\dd}{\mathop{}\!\mathrm{d}}

\newcommand{\eps}{\ensuremath{\varepsilon}}

\newcommand{\ceq}{\coloneqq}

\newcommand{\R}{{\mathbb{R}}}

\newcommand{\N}{\mathbb{N}}

\DeclarePairedDelimiterXPP{\pars}[1]{\mathop{}}{\lparen}{\rparen}{}{#1}
\DeclarePairedDelimiterXPP{\abs}[1]{\mathop{}}{\lvert}{\rvert}{}{#1}
\DeclarePairedDelimiterXPP{\norm}[1]{\mathop{}}{\lVert}{\rVert}{}{#1}
\DeclarePairedDelimiterXPP{\seminorm}[1]{\mathop{}}{\lbrack}{\rbrack}{}{#1}
\DeclarePairedDelimiterXPP{\inner}[1]{\mathop{}}{\langle}{\rangle}{}{#1}
\DeclarePairedDelimiterXPP{\iinner}[1]{\mathop{}}{\langle\!\langle}{\rangle\!\rangle}{}{#1}
\DeclarePairedDelimiterXPP{\brackets}[1]{\mathop{}}{\lbrack}{\rbrack}{}{#1}
\DeclarePairedDelimiterXPP{\braces}[1]{\mathop{}}{\lbrace}{\rbrace}{}{#1}

\DeclarePairedDelimiterXPP{\floor}[1]{\mathop{}}{\lfloor}{\rfloor}{}{#1}
\DeclarePairedDelimiterXPP{\ceil}[1]{\mathop{}}{\lceil}{\rceil}{}{#1}

\DeclarePairedDelimiterXPP{\intervalcc}[1]{\mathop{}}{\lbrack}{\rbrack}{}{#1}
\DeclarePairedDelimiterXPP{\intervalco}[1]{\mathop{}}{\lbrack}{\rparen}{}{#1}
\DeclarePairedDelimiterXPP{\intervaloc}[1]{\mathop{}}{\lparen}{\rbrack}{}{#1}
\DeclarePairedDelimiterXPP{\intervaloo}[1]{\mathop{}}{\lparen}{\rparen}{}{#1}

\DeclarePairedDelimiterXPP{\myset}[2]{\mathop{}}{\lbrace}{\rbrace}{}{#1\,\delimsize\vert\,\mathopen{}#2}

\DeclareDocumentCommand{\Graph}{ O{} O{} o o}{
	\IfValueTF{#3}{
	  \IfValueTF{#4}{
	  	\mathrm{Graph}^{#1}_{#2}(#3;#4)
	  }{
		\mathrm{Graph}^{#1}_{#2}(#3)
	  }
	}{
		\mathrm{Graph}^{#1}_{#2}
	}
}

\DeclareDocumentCommand{\Emb}{ O{} O{} o o}{
	\IfValueTF{#3}{
	  \IfValueTF{#4}{
	  	\mathrm{Emb}^{#1}_{#2}(#3;#4)
	  }{
		\mathrm{Emb}^{#1}_{#2}(#3)
	  }
	}{
		\mathrm{Emb}^{#1}_{#2}
	}
}

\DeclareDocumentCommand{\Sobo}{ O{} O{} o o}{
	\IfValueTF{#3}{
	  \IfValueTF{#4}{
	  	W^{#1}_{#2}(#3;#4)
	  }{
	  	W^{#1}_{#2}(#3)
	  }
	}{
	  W^{#1}_{#2}
	}
}

\DeclareDocumentCommand{\Bessel}{ O{} O{} o o}{
	\IfValueTF{#3}{
	  \IfValueTF{#4}{
	  	H^{#1}_{#2}(#3;#4)
	  }{
	  	H^{#1}_{#2}(#3)
	  }
	}{
	  H^{#1}_{#2}
	}
}

\DeclareDocumentCommand{\Holder}{ O{} O{} o o}{
	\IfValueTF{#3}{
	  \IfValueTF{#4}{
	  	C^{#1}_{#2}(#3;#4)
	  }{
	  	C^{#1}_{#2}(#3)
	  }
	}{
	  C^{#1}_{#2}
	}
}
\DeclareDocumentCommand{\HolderC}{ O{} O{} }{\Holder[#1][#2][\Circle][\AmbSpace]}

\DeclareDocumentCommand{\Lebesgue}{ O{} O{} o o}{
	\IfValueTF{#3}{
	  \IfValueTF{#4}{
	  	L^{#1}_{#2}(#3;#4)
	  }{
	  	L^{#1}_{#2}(#3)
	  }
	}{
	  L^{#1}_{#2}
	}
}
\DeclareDocumentCommand{\LebesgueC}{ O{} O{} }{\Lebesgue[#1][#2][\Circle][\AmbSpace]}

\usepackage[capitalise,nameinlink]{cleveref}

\crefname{equation}{}{}
\newcommand{\aref}[1]{\cref{#1}}
\newtheorem{theorem}{Theorem}[section]

\newtheorem{proposition}[theorem]{Proposition}

\theoremstyle{break}

\theoremstyle{plain}
\theorembodyfont{\normalfont}
\newtheorem{definition}[theorem]{Definition}
\newtheorem{example}[theorem]{Example}
{%
    \theoremsymbol{\ensuremath{\Diamond}}%
    \newtheorem{remark}[theorem]{Remark}%
}

\theoremstyle{break}

\theoremheaderfont{\itshape}
\theorembodyfont{\upshape}
\theoremstyle{nonumberplain}
\theoremseparator{.}
\theoremsymbol{\ensuremath{\Box}}	
\newtheorem{proof}{Proof}

\title%[On a Complete Riemannian Metric on the Space of Embedded Curves]
{A Lusin theorem for nonlocal gradients}

\author[1]{Elias Döhrer}

\affil[1]{Chemnitz University of Technology}

\begin{document}

\maketitle

We extend the celebrated result of Alberti, 
stating that Borel vector fields coincide with gradients of $C^1$-functions outside of a set of arbitrary small measure.
We prove that a similar statement holds true in the setting of fractional gradients and $C^{0,s}$-functions.
Furthermore, we establish analogous results for general nonlocal gradients with appropriate function spaces.
Both statements are proven by means of the translation method, demonstrating its use for the purposes of geometric measure theory.
In particular, this article illustrates how the translation method transfers rigidity phenomena from the classical gradient to a broad class of nonlocal gradients.

\tableofcontents

\section{Introduction}
A fundamental question in geometric measure theory asks to what extent arbitrary vector fields can be realized as gradients of sufficiently regular functions. 
Alberti's celebrated theorem provides a striking answer: every Borel vector field coincides with the gradient of a $C^1$-function outside a set of arbitrarily small measure. 
This result lies at the interface of differentiability theory, geometric measure theory and the calculus of variations.
There are various refinements in the direction of more general measures~\cite{DeMasiMarchese-RefinedLusinForGradients} and extensions to higher dimensional cases~\cite{MoonensPfeffer-MultidimensionalLusin}.
\\
In this article, we extend Alberti's theorem~\cite{LusinForGradients-Alberti} to the setting of fractional and nonlocal gradients.
We consider classical fractional gradients, defined via the Riesz transform, fractional gradients with finite horizon and more generally nonlocal gradients, defined via convolution with singular kernels.
These objects have become of great interest in recent years and their use has been demonstrated in various cases.
Initially, they were developed in the context of \textit{peridynamics}, which is an approach to solid mechanics using nonlocal quantities.
This approach experienced great progress within the last decades, see~\cite{Silling-Peridynamics} for the original work and~\cite{AnotherLookArSoboSpaces-BBM,MengeshaDu-CharFuncSpacesOfVFByNonlinPeridynamics,BellidoMorraCOrral-ExNonlocVarProbinPeridynamics} for a recent mathematical exhibition.
Subsequently, in~\cite{zbMATH07208002}, the authors modified the classical approach by using Riesz fractional gradients in order to describe the model of bond based peridynamics.
It was Shieh and Spector who drew the attention of the PDE and Calculus of Variations community towards these models, cf.~\cite{ShiehSpecotr-NewFractPDE1,ShiehSpector-NewFractPDE2}.
In their seminal works, they proved that the Bessel potential spaces are in fact the spaces associated to the Riesz fractional gradients.
For $s\in (0,1), u\in C^\infty_c(\R^n)$, the Riesz fractional gradient is given by
\begin{equation}
    D^s u(x)\ceq
    c_{n,s} \int_{\R^n}
        \frac{u(y)-u(x)}{\abs{y-x}}\frac{y-x}{\abs{y-x}^{n-1+s}}
    \dd y.
\end{equation}
These core insights motivated the theory of nonlocal gradients, leading to a more detailed understanding of the spaces~\cite{ComiStefani-Asymptotics-DistrApproachToFractionalSobolevSpaces,COMI20193373} and solution theory of the according variational problems~\cite{zbMATH07439869}.
Furthermore, the nonlocal spaces have been related to the local spaces using $\Gamma$-convergence, cf.~\cite{zbMATH07296599}.
\\
In order to work in bounded domains, one introduced kernels with finite horizon, circumventing the integration over the whole space which is required by classical Riesz fraction gradients, cf.~\cite{BellidoCuetoMoraCorral-FundamentalTheoremNonlocalGradients,CuetoKreisbeckSchoenberger-VarTheoryIntegralFunctFiniteHorizon}.
This can be done by multiplying with an admissible cut-off $\omega_\delta$, supported in $B_\delta(0)$.
The fractional gradient with horizon $\delta>0$ is given by
\begin{equation}
    D^s_\delta u(x)\ceq
    c_{n,s} \int_{\R^n}
        \frac{u(y)-u(x)}{\abs{y-x}}\frac{y-x}{\abs{y-x}^{n+s}}
        \omega_\delta(y-x)
    \dd y.
\end{equation}
The main tools for calculus of variations with nonlocal gradients are the translation method and the nonlocal fundamental theorem of calculus, developed in~\cite{BellidoMorraCorralSchoenberger-NonLocGrad-FundamentalPoincareEmbeddings,CuetoKreisbeckSchoenberger-VarTheoryIntegralFunctFiniteHorizon}.
Naturally, this led to considering more general kernels, admissible kernels $\rho$, giving rise to the nonlocal gradient with respect to $\rho$
\begin{equation}
    D^\rho u(x)
    \ceq
    \int_{\R^n}
    \frac{u(y)-u(x)}{\abs{y-x}}
    \frac{y-x}{\abs{y-x}}
    \rho(y-x)
    \dd y.
\end{equation}
The theory of fractional gradients with finite horizon extends almost one to one to the case of more general kernels, as one can see in~\cite{NonlocPerimetersVariations:ExtremalityAndDecomposability-CarioniDelFrandeIglesiasSchoenberger,BellidoMorraCorralSchoenberger-NonLocGrad-FundamentalPoincareEmbeddings,CuetoKreisbeckSchoenberger-VarTheoryIntegralFunctFiniteHorizon} 
and many other works. \\
The main contribution of this article is to answer whether Alberti's result extends to the setting of nonlocal gradients affirmatively. 
We first establish a Lusin-type theorem for Riesz fractional gradients and Hölder functions. 
Using the translation method, we subsequently extend the result to a broad class of nonlocal gradients. 
We further obtain BV-versions of these results and investigate fine differentiability properties of the associated Sobolev–Slobodeckij spaces.

\subsection{Outline of the paper}
In~\cref{sec:Prelim}, we recall the basic definitions needed for nonlocal gradients and the result of Alberti.
Subsequently, we prove the different generalization of Alberti's result in~\cref{sec:LusinNonlocalGradient}.
Ultimately, we investigate fine properties of function belonging to the Sobolev-Slobodeckij spaces associated to $\rho$, namely $L^p$ and approximate differentiability.
We also draw a connection to the classical Sobolev-Slobodeckij spaces.

\section{Preliminaries}\label{sec:Prelim}
In this chapter, we clarify the general setting, starting with admissible kernels and nonlocal gradients.
Afterward, we recall the tools, which are important for our purposes.\\
For a radial function $\rho:\R^n\setminus\{0\}\rightarrow \R$, we define $\overline{\rho}:(0,\infty)\rightarrow\intervalco{0,\infty}$ by $\overline{\rho}(\abs{x})=\rho(x)$.
\begin{definition}
    Let $\rho:\R^n\setminus \{0\}\rightarrow \intervalco{0,\infty}$ be a radial function. 
    We say $\rho$ is an \textit{admissible kernel} if there are $\eta, \nu>0, 0<\sigma \leq \gamma <1$ such that $\rho$ satisfies the following properties.
    \begin{itemize}
        \item[(H0)] $\inf_{\overline{B_\eta(0)}}\rho >0$ and $\pars*{\rho(\cdot )\min(1, \frac{1}{\abs{\cdot}})}\in L^1(\R^n)$
        \item[(H1)] the maps $f_\rho:r\mapsto r^{n-2} \overline{\rho}(r)$ and $g_{\rho}:r\mapsto r^{\nu+n-2} \overline{\rho}(r)$ are decreasing on $(0,\eta)$
        \item[(H2)] $f_\rho \in C^\infty((0,\infty))$ and for all $k\in \N$, there is $C(k)>0$ such that
        \[
            \abs{\frac{\dd^k}{\dd r^k}f(r)}\leq C(k)r^k f_\rho(r)
        \]
        \item[(H3)] $r\mapsto r^{n+\sigma-1}\overline{ \rho}(r)$ is almost decreasing on $(0,\eta)$
        \item[(H4)] $r\mapsto r^{n+\gamma-1}\overline{\rho}(r)$ is almost increasing on $(0,\eta)$
    \end{itemize}
We denote the set of all admissible kernels by $\mathcal{A}$.
For $\rho\in \mathcal{A}, u\in C^\infty_c(\R^n)$, we define the \textit{nonlocal gradient with kernel $\rho$ of u} as 
\begin{equation}\label{eqn:DefNonLocGradRho}
    D^\rho u(x)\ceq 
        \int_{\R^n}
            \frac{u(y)-u(x)}{\abs{y-x}} \frac{y-x}{\abs{y-x}}\rho(y-x)
        \dd y.
\end{equation}
\end{definition}
Note that by assuming (H2), the properties (H3) and (H4) are satisfied everywhere on $(0,\eta)$. 
\begin{example}
    We provide an incomplete list of examples of admissible kernels.
    For $s\in (0,1)$, let $\rho^s\ceq c_{n,s} \frac{1}{\abs{\cdot}^{n+s-1}}\in \mathcal{A}$, where $c_{n,s}$ is a normalization constant.
    Then $D^s\ceq D^{\rho^s}$ is the well-known Riesz-fractional gradient.\\
    In fact, if $s\in C^\infty(\intervalco{0,\infty}, (0,1))$, with $0<\inf s(x) \leq \sup s(x)<1$ and bounded derivatives, then 
    \[
        \rho(x)\ceq \frac{1}{\abs{x}^{n+s(\abs{x})-1}}
    \] is an admissible kernel.
    Furthermore, $\mathcal{A}$ admits a cone structure, i.e.\ for $\rho_1, \rho_2\in \mathcal{A}, \lambda_1,\lambda_2 >0$, one observes $\lambda_1 \rho_1+\lambda_2 \rho_2 \in \mathcal{A}$.
\end{example}

Additionally, we define admissible cut-offs for kernels $\rho\in \mathcal{A}$, giving notion to fractional gradients with finite horizons.
The following conditions have been given in~\cite{CuetoKreisbeckSchoenberger-VarTheoryIntegralFunctFiniteHorizon}.
\begin{definition}
    Let $\delta>0$ be given.
    A function $\omega_\delta :\R^n \rightarrow \R$ is an \textit{admissible cut-off with horizon $\delta$} if it satisfies the following properties.
    \begin{itemize}
        \item[(i)] $\omega_\delta$ is radial, i.e.\ there exists $\overline{\omega}_\delta :\intervalco{0,\infty}\rightarrow \R$ such that $\omega_\delta(x)= \overline{\omega}_\delta(\abs{x})$
        \item[(ii)] $\omega_\delta \in C^\infty_c (B_\delta(0))$
        \item[(iii)] $\exists b_0 \in (0,1)$ such that $\omega_\delta(x)=1$ for all $x\in B_{b_0\delta}(0)$
        \item[(iv)] $\omega_\delta$ is radially decreasing, i.e. $\overline{\omega}_\delta(r_1)\geq \omega_\delta(r_2)$ for all $0\leq r_1 \leq r_2 <\infty$
    \end{itemize}
    The set of admissible cut-offs with horizon $\delta$ will be denoted by $\mathcal{C}_\delta$.\\
    For $\omega_\delta \in \mathcal{C}_\delta, s\in (0,1)$, we define the \textit{fractional gradient with horizon $\delta$} of $u\in C^\infty_c$ as
    \begin{equation}
        D^s_\delta u(x)
        \ceq 
        c_{n,s} \int_{\R^n}\frac{u(y)-u(x)}{\abs{y-x}^{n+s}}\frac{y-x}{\abs{y-x}}\omega_\delta(y-x) \dd y
        .
    \end{equation}
\end{definition}
Naturally, one defines the according vector operations.
For $\rho\in \mathcal{A}, \varphi\in C^\infty_c(\R^n, \R^n)$, we define the \textit{nonlocal divergence with kernel $\rho$} as 
\begin{equation}
    \text{div}^\rho \varphi(x) \ceq \int_{\R^n}\inner{ \frac{\varphi(y)-\varphi(x)}{\abs{y-x}}, \frac{y-x}{\abs{y-x}}}\rho(y-x)\dd y
    .
\end{equation}
If $\rho=\rho^s=c_{n,s} \frac{1}{\abs{\cdot }^{n+s-1}}$, we abbreviate $\text{div}^s\varphi(x)= \text{div}^{\rho^s} \varphi(x)$.
Recall, that for $f:\R^n\rightarrow \R^n$, the curl is defined as
\[
    \mathrm{curl}(f)=( \partial_j f_i - \partial_i f_j)_{i,j=1}^n
    =
    ( (\nabla f_i)_j - (\nabla f_j)_i)_{i,j=1}^n
.
\]
This definition naturally extends to the nonlocal and fractional setting, by 
\[
    \mathrm{curl}^\rho(f)
    =
    ( (D^\rho f_i)_j - (D^\rho f_j)_i)_{i,j=1}^n
    .
\]
In the case $n=3$ and $\rho(\cdot)= c_{3,s,\delta} \frac{w_\delta}{\abs{\cdot}^{n+s-1}}$, a short computation reveals that our definition of $ \mathrm{curl}^\rho$ is equivalent to the definition given in~\cite{BellidoCuetoFossRadu-NonLocalGreenHelmholtz}.
Having collected all the above definitions, we can define the fractional Sobolev spaces, cf.~\cite{BellidoMorraCorralSchoenberger-NonLocGrad-FundamentalPoincareEmbeddings,COMI20193373,ComiStefani-Asymptotics-DistrApproachToFractionalSobolevSpaces} for a more detailed exposition and a distributional approach.
\begin{definition}
    For $\rho\in \mathcal{A}$, $p\in \intervalcc{1,\infty}$, we define 
    \[
        H^{\rho,p}(\R^n)\ceq \{ u\in L^p(\R^n): D^\rho u\in L^p(\R^n)\}
    \]
    and equip it with the norm $\norm{u}_{H^{\rho,p}(\R^n)} \ceq \norm{u}_{L^p(\R^n)}+ \norm{D^\rho u}_{L^p(\R^n)}$.
    Furthermore, for a domain $\Omega\subset\R^n$ with Lipschitz boundary, we define
    \[
        H^{\rho,p}_0(\Omega)\ceq
        \{
            u\in H^{\rho,p}(\R^n): u=0 \text{ almost everywhere in }\Omega^c    
        \}
        .
    \]
    For $\rho\in \mathcal{A}$ with finite horizon $\delta>0$, let 
    \[
        H^{\rho,p,\delta}(\Omega)\ceq
        \{
            u\in L^p(\Omega_\delta): D^{\rho}_{\delta} u \in L^p(\Omega)  
        \}
        ,
        \text{ where }
        \Omega_\delta\ceq \Omega +B_\delta(0)
    .
    \]
    If $\rho= \rho^s$ or $\rho= \rho^s \omega_\delta$, for $\omega_\delta \in \mathcal{C}_\delta$, 
    the nonlocal Sobolev space over $\R^n$ coincides with the Bessel potential space~\cite{CuetoKreisbeckSchoenberger-VarTheoryIntegralFunctFiniteHorizon}.
    Hence, we abbreviate
    $H^{s,p}\ceq H^{\rho^s,p}$ and $H^{s,p,\delta}\ceq H^{\rho^s \omega_\delta, p,\delta}$.
\end{definition}
Using Fourier analysis, one can show that 
\[
    H^{\rho,p}_0(\Omega)= 
    \overline{C^\infty_c(\Omega)}^{\norm{\cdot}_{H^{\rho,p}(\Omega)}} 
    \text{ for }
    1<p<\infty.
\]
For a detailed exposition of the spaces $H^{\rho,p,\delta}(\Omega)$, $H^{s,p}(\Omega)$ and $H^{s,p,\delta}(\Omega)$, we refer to~\cite{BellidoMorraCorralSchoenberger-NonLocGrad-FundamentalPoincareEmbeddings,CuetoKreisbeckSchoenberger-VarTheoryIntegralFunctFiniteHorizon,NonConstFunctionWithZeroNonLocGradient-KreisbeckSchoenberger}.
Note, that for $\rho\in \mathcal{A}$, one can also define the Sobolev--Slobodeckij space w.r.t $\rho$.
For $p\in \intervalcc{1,\infty}$, these spaces can be defined as
\[
    W^{\rho,p}(\Omega)\ceq 
    \{
        u\in L^p(\Omega): 
        \norm*{
            \frac{(u(y)-u(x))\rho(x-y)}{\abs{y-x}^{1-n}}
        }_{L^p(\Omega\times \Omega, \frac{\dd x \dd y}{\abs{y-x}^n})}<\infty
    \}
.
\]
We define 
\[
    \seminorm{u}_{W^{\rho,p}(\Omega)}\ceq \norm*{
            \frac{(u(y)-u(x))\rho(x-y)}{\abs{y-x}^{1-n}}
        }_{L^p(\Omega\times \Omega, \frac{\dd x \dd y}{\abs{y-x}^n})}
\]
Subsequently, $\norm{u}_{W^{\rho,p}(\Omega)}\ceq \pars*{\norm{u}_{L^p(\Omega)^p}+\seminorm{u}_{W^{\rho,p}(\Omega)}^p}^{1/p}$.
In the case $p=1$, this definition coincides with the one given in~\cite[§2]{NonlocPerimetersVariations:ExtremalityAndDecomposability-CarioniDelFrandeIglesiasSchoenberger}.
A more detailed exhibition of the corresponding $BV$ and Sobolev--Slobodeckij spaces is given in~\cite{NonlocPerimetersVariations:ExtremalityAndDecomposability-CarioniDelFrandeIglesiasSchoenberger} and~\cite{ComiStefani-Asymptotics-DistrApproachToFractionalSobolevSpaces}.
\\
As mentioned in the introduction, one of the main tools is the so-called translation method.
In the case of Riesz fractional gradients, one can ``decompose'' the ordinary gradients into a fractional gradient and the fractional Laplacian.
Shieh and Spector~\cite{ShiehSpecotr-NewFractPDE1,ShiehSpector-NewFractPDE2} have characterized the fractional gradient as a composition of the ordinary gradient after the convolution with the Riesz potential kernel $\mathcal{I}_{1-s} \ceq c_{n,s} \abs{\cdot}^{1-(n+s)}$.
Note that $\mathcal{I}_{s} \ast u = (-\Delta)^{-s/2} u$ for all $u\in C^\infty_c(\R^n)$.
The preceding facts can be summarized as
\begin{equation}
    \nabla u= D^s \circ (-\Delta)^{\frac{1-s}{2}} u
    \text{ and }
    D^s u= \nabla \circ \mathcal{I}_{1-s} \ast u= \mathcal{I}_{1-s} \ast D u
    \text{ for all }
    u\in C^\infty_c(\R^n)
    .
\end{equation}
Subsequently, this translation from the local to the nonlocal setting and vice versa was extended to nonlocal gradients.
For $\rho \in \mathcal{A}$ and $u\in C^\infty_c$, we define
\begin{align*}
    &Q_\rho(x)\ceq \int_{\abs{x}}^\infty \frac{\overline{\rho}(t)}{t}\dd t
    \text{ and }
    \mathcal{Q}_\rho u= Q_\rho \ast u
    .
\end{align*}
As proven in~\cite{BellidoMorraCorralSchoenberger-NonLocGrad-FundamentalPoincareEmbeddings}, the operator $\mathcal{Q}_\rho$ extends to an operator $H^{\rho,p}(\R^n)\rightarrow H^{1,p}(\R^n), p\in (1,\infty)$, and satisfies
\begin{equation}
    D^\rho \phi= \mathcal{Q}_\rho \nabla \phi= \nabla \mathcal{Q}_\rho \phi
    \text{ for all }
    \phi \in C^\infty_c
    .
\end{equation}
In pursuit of defining an inverse to $\mathcal{Q}_\rho$, they defined 
\begin{equation}\label{eq:Intro-PRho}
        P_\rho= (1/ \widehat{Q_\rho})^{\vee}
        \text{ and } 
        \mathcal{P}_\rho u= P_\rho \ast u,
\end{equation}
where $\widehat{\cdot}$ and $\cdot^{\vee}$ denote the Fourier transform and the inverse Fourier transform~\cite{GammaConvNonLocGrad-CuetoKreisbeckSchoenberger}.
$\mathcal{P_\rho}$ satisfies $D\phi= \mathcal{P}_\rho D^\rho \phi= D^\rho \mathcal{P}_\rho \phi$.
In the same work, the authors showed that $\mathcal{P}_\rho$ extends to an operator $H^{1,p}(\R^n)\rightarrow H^{\rho,p}(\R^n)$ for $1<p<\infty$.
The translation method motivated a nonlocal fundamental theorem of calculus, where one finds a vector-valued kernel $V_\rho$ such that $u= V_\rho \ast D^\rho u$, cf.~\cite{BellidoCuetoMoraCorral-FundamentalTheoremNonlocalGradients}.
For finite horizon kernels, this method extends to bounded domains, giving rise to a nonlocal Green theorem, cf.~\cite{BellidoCuetoFossRadu-NonLocalGreenHelmholtz}
--- at the expense of counterintuitive fact that there are nonconstant functions with zero nonlocal gradient, see~\cite{NonConstFunctionWithZeroNonLocGradient-KreisbeckSchoenberger}.
\\
Exploiting the above translation mechanism, one verifies that for $\varphi \in C^\infty_c(\R^n, \R^n)$
\begin{equation}\label{eq:divRhoVectorCalc}
    \mathrm{div}^\rho \varphi
    =
    \trace( D^\rho \varphi)
    =
    \trace( \nabla (Q_\rho \ast \varphi))
    =
    \trace(Q_\rho \ast \nabla \varphi)
    =
    Q_\rho \ast \mathrm{div}(\varphi).
\end{equation}
Additionally, using the fundamental theorem of nonlocal calculus,~\cite[Proposition 4.6]{NonlocPerimetersVariations:ExtremalityAndDecomposability-CarioniDelFrandeIglesiasSchoenberger} establishes that
\begin{equation}\label{eq:NonlocVC-divRhoVRho}
    \mathrm{div}^\rho (V_\rho \ast \varphi)= \varphi \text{ for all } \varphi \in C^\infty_c(\R^n).
\end{equation}
\noindent
We conclude this section by recalling the celebrated result of Alberti from~\cite{LusinForGradients-Alberti}.
For a Borel measurable set $A\subset \R^n$, we denote the Lebesgue measure of $A$ by $\abs{A}$.
\begin{theorem}[{\cite%[Theorem 1]
    {LusinForGradients-Alberti}}]
\label{thm:Alberti-Lusin4Grad}
    Let $\Omega\subset \R^n$ be an open set of finite measure and $f:\Omega \rightarrow \R^n$ a Borel function.
    Then for any $\eps>0$, there is an open set $A\subset \Omega$ and $u\in C^{1}_0(\Omega)$ such that
    \begin{align*}
        &\abs{ A}\leq \eps \abs{\Omega},
        \\
        &f= \nabla u \text{ in } \Omega \setminus A,\\
        &\norm{\nabla u}_{L^p} \leq c(n) \eps ^{1/p-1} \norm{f}_{L^p}
        \text{ for all }p\in \intervalcc{1,\infty}.
    \end{align*}
\end{theorem}
\begin{remark}\label{rmk:Alberti-C1c}
    Since $\Omega$ is open and has finite measure, we may approximate the domain from the inside by sets $U_n$ having compact closure in $\Omega$.
    Hence, we may assume that $u\in C^1_c(\Omega)$.
    This has been proven in~\cite[Corollary 1.2]{MoonensPfeffer-MultidimensionalLusin}.
\end{remark}    
Furthermore, in~\cite[Remark 2]{LusinForGradients-Alberti}, Alberti states that for $p=1$, one can drop the assumption $\abs{\Omega}<\infty$.
The above theorem has been generalized in various ways.
Firstly, Moonens and Pfeffer extended the whole framework to charges and higher dimensions in~\cite{MoonensPfeffer-MultidimensionalLusin,zbMATH02023430}.
More recently, De Masi and Marchese proved a refined version of the above for varifolds, cf.~\cite{DeMasiMarchese-RefinedLusinForGradients}.
\section{Lusin for nonlocal gradients}\label{sec:LusinNonlocalGradient}
In this section, we extend Alberti's theorem to nonlocal gradients.
The general approach is to use \aref{thm:Alberti-Lusin4Grad} and the translation method as often as possible.
It is crucial to ensure that all occurring quantities are well-defined.
Throughout this section, we fix 
\[
    s\in (0,1).
\]
\begin{theorem}\label{thm:LusinFractGradient}
    Let $\Omega\subset \R^n$ be an open set with finite measure and $f:\Omega\rightarrow \R^n$ a Borel vector field.
    Then, for all $\eps>0$ there is an open $A\subset \Omega$ and $v\in C^{0,s}(\R^n, \R)$ such that
    \begin{equation}  
        \begin{split}
        &\abs{A}\leq \eps \abs{\Omega},\\
        &f= D^s v \;\text{ in }\Omega \setminus A,\\
        &\norm{D^s v}_{L^p(\R^n)}\leq 
        c(n)
        \eps^{1/p-1} \norm{f}_{L^p(\Omega)} 
        \text{ for all } p\in \intervalcc{1,\infty}.
        \end{split}
    \end{equation}
\end{theorem}
\begin{proof}
    For given $\eps>0$, we apply~\cite{LusinForGradients-Alberti} and \aref{rmk:Alberti-C1c} in order to find $u\in C^1_c(\Omega)$ and an open set $A\subset \R^n$ such that
    \begin{align*}
        &\abs{A}\leq \eps \abs{\Omega},\\
        &f= \nabla u \text{ in }\Omega \setminus A,\\
        &\norm{\nabla u}_{L^p}\leq c(n) \eps^{1/p-1} \norm{f}_{L^p}.
    \end{align*}
    We extend $u$ by $0$ from $\Omega$ to $\R^n$.
    Since $u\in C^1_c$, one concludes that the extension is in $W^{1,\infty}_c =\text{Lip}_c$.
    By~\cite{ShiehSpecotr-NewFractPDE1,ShiehSpector-NewFractPDE2}, we can decompose $D^s=\nabla\circ \mathcal{I}_{1-s}$ on $C^\infty_c$.
    Using~\cite[Proposition 2.1]{ComiStefani-Asymptotics-DistrApproachToFractionalSobolevSpaces}, this extends to $\mathrm{Lip}_c(\R^n)$.
    Hence, we can apply it to $u$ and observe 
    \[
        \nabla u
        =
        \nabla \circ \mathcal{I}_{1-s}\circ (-\Delta)^{\frac{1-s}{2}}u
        = D^s (-\Delta)^{\frac{1-s}{2}}u
        .
    \]
    Furthermore, by~\cite[Proposition 2.5]{SilvestreRegularityObstacle}, we infer that $(-\Delta)^{\frac{1-s}{2}}u \in C^{0,s}(\R^n)$.
    Letting $v\ceq (-\Delta)^{\frac{1-s}{2}}u$, we conclude that
    \[
        D^s v= D^s(-\Delta)^{\frac{1-s}{2}}u
        =\nabla u =f \text{ on }\Omega \setminus A
        \text{ and }
        \abs{A}\leq \eps \abs{\Omega}
    .
    \]
    Furthermore, we estimate
    \begin{align*}
        \norm{D ^s v}_{L^p(\R^n)}
    &=
        \norm{D^s  (-\Delta)^{\frac{1-s}{2}} u}_{L^p(\R^n)}
    =
        \norm{\nabla u}_{L^p(\Omega)}
    \leq
        c(n) \eps^{1/p -1} \norm{f}_{L^p(\Omega)}.
    \end{align*}
\end{proof}
\begin{remark}
    Contrary to~\cite[Remark 2]{LusinForGradients-Alberti}, we cannot drop the finite measure assumption on $\Omega$ for $p=1$.
    This stems from the problem that $u\in C^1_0(\Omega)$, but it is not compactly supported.
    On $\text{Lip}_{b}(\R^n)$, we can use $\nabla= D^s \circ (-\Delta)^{\frac{1-s}{2}}$.
    If we drop the claim $\abs{\Omega}<\infty$, we do not obtain compact support or global boundedness.
    Therefore, $u\in \text{Lip}_{\text{loc}}(\Omega)$.
    Since we lack information on $\partial \Omega$, we cannot modify the construction in order to guarantee $u\in \text{Lip}_b$. 
    Hence, we cannot rely on the translation method.
\end{remark}
If we consider the same kind of problem with nonlocal gradients of finite horizon, we obtain a somewhat ``more local'' version of the last result.
More precisely, we do not have to define $v$ on $\R^n$ but just on the thickened domain $\Omega_\delta \ceq \Omega + B_\delta(0)$.
In order to investigate more general kernels, we have to extend the translation method to $p=\infty$.

\begin{theorem}\label{thm:TranslationMechNonlocGrad-inftyCase}
    Let $\rho\in \mathcal{A}$ with compact support.
    The following statements hold true.
    \begin{itemize}
        \item[(a)] $\mathcal{Q}_\rho: H^{\rho, \infty}(\Omega)\rightarrow W^{1,\infty}(\Omega), u\mapsto Q_\rho \ast u$, is a bounded operator.
                    If $u\in H^{\rho,\infty}(\Omega)$ then $\nabla (\mathcal{Q}_\rho u)=D^\rho u$.
        \item[(b)] The operator $\mathcal{P}_\rho : \mathcal{S}(\R^n)\rightarrow  \mathcal{S}(\R^n), \phi \mapsto (\widehat{\phi}/\widehat{\mathcal{Q}_\rho})^{\vee}$ is well-defined and extends to a bounded linear operator $W^{1,\infty}(\R^n)\rightarrow H^{\rho,\infty}(\R^n)$ such that $\mathcal{P}_\rho=(\mathcal{Q}_\rho)^{-1}$.
                    In particular, for $u\in H^{\rho,\infty}(\R^n), v\in W^{1,\infty}(\R^n)$, one has $\nabla v= D^\rho \mathcal{P_\rho}v$ and $\mathcal{Q}_\rho \mathcal{P}_\rho v= v, \mathcal{P}_{\rho}\mathcal{Q}_\rho u=u$.
    \end{itemize}
\end{theorem}
\begin{proof}
    We start by proving the first statement.
    Since $\rho$ has compact support, there exists $R>1$ such that $\mathrm{spt}(\rho)\subset B_R(0)$.
    Hence, 
    \[
        \rho(\cdot) \frac{1}{R}\leq \rho(\cdot) \min(1, \frac{1}{\abs{\cdot}})\in L^1(\R^n).
    \]
    Invoking (H0), we observe that $\rho \in L^1(\R^n)$ and has compact support.
    By~\cite[Lemma 2.5]{BellidoMorraCorralSchoenberger-NonLocGrad-FundamentalPoincareEmbeddings},
    we conclude that $Q_\rho \in \Lebesgue[1](\R^n)$ has compact support.
    Let $u\in H^{\rho, \infty}(\R^n)$ be arbitrary.
    By Young's convolution inequality, one has that $\mathcal{Q}_\rho u= Q_\rho \ast u \in L^\infty(\Omega)$.
    By the virtue of~\cref{eq:divRhoVectorCalc}, we observe that for $\phi \in C^\infty_c(\R^n,\R^n)$
    \begin{equation}\label{eq:IntByPartsDivQ}
        \int_\Omega
            \mathcal{Q}_\rho u \mathrm{div}\phi 
        \dd x
        =
        \int_{\Omega}
            u ( Q_\rho \ast \mathrm{div}\phi)
        \dd x
        =
        \int_{\Omega}
            u \mathrm{div}^\rho \phi
        \dd x
        =
        -\int_{\Omega}
            \phi D^\rho u 
        \dd x
    .
    \end{equation}
    The above implies $\mathcal{Q}_\rho u \in W^{1,\infty}(\R^n)$ and $\nabla (\mathcal{Q}_\rho u)= D^\rho u$.
    Putting all the above together, we conclude that
    \[
        \norm{\mathcal{Q}_\rho u}_{W^{1,\infty}(\R^n)}
        \leq
        \norm{Q_\rho}_{\Lebesgue[1](\R^n)}\norm{u}_{\Lebesgue[\infty](\R^n)}
        +
        \norm{D^\rho u}_{\Lebesgue[\infty](\R^n)}
        .
    \]
    In particular, the operator $\mathcal{Q}_\rho :H^{\rho,\infty}(\Omega)\to W^{1,\infty}(\R^n)$ is bounded.
    \\
    We now verify the second part of the theorem.
    By~\cite[Lemma 2.6]{GammaConvNonLocGrad-CuetoKreisbeckSchoenberger}, one observes that $\widehat{Q_\rho}$ is smooth, radial and positive.
    We define
    \[
        \mathcal{P}_\rho: \mathcal{S}(\R^n)\rightarrow \mathcal{S}(\R^n),
        \quad 
        v\mapsto ( \widehat{v}/ \widehat{Q_\rho})^{\vee}.
    \]
    Therefore, $\mathcal{P}_\rho= \mathcal{Q}_\rho ^{-1}$ on the class of Schwarz functions.
    We now show that this holds on the Lipschitz functions.
    Due to part (a) of this theorem and the Banach isomorphism theorem, it suffices to show that $\mathcal{Q}_\rho$ is bijective.
    Let  $\mathcal{Q}_\rho u=0$ for $u\in H^{\rho, \infty}(\R^n)$.
    Therefore, the weak nonlocal gradient of $u$ vanishes, i.e.
    \[
        D^\rho u= \nabla  \mathcal{Q}_\rho u =0
    .
    \]
   The definition of the weak nonlocal gradient implies that $\int_{\R^n} u \mathrm{div}^\rho \phi \dd x=0$ for all $\phi \in C^\infty_c(\R^n, \R^n)$, cf.~\cite[Definition 2.3]{GammaConvNonLocGrad-CuetoKreisbeckSchoenberger}.
   By density, this extends to $\phi \in \mathcal{S}(\R^n, \R^n)$.
   Let $\psi \in C^\infty_c (\R^n)$ be arbitrary and define $\phi\ceq \mathcal{P}_\rho \psi\in \mathcal{S}(\R^n,\R^n)$.
   One computes
   \begin{align*}
        0&=
        \int_{\R^n} u\;\mathrm{div}^\rho \phi \dd x
        =
        \int_{\R^n} u \;\mathrm{div}^\rho \mathcal{P}_\rho \psi \dd x
        =
        \int_{\R^n} u\; \mathrm{div} \mathcal{Q}_\rho \mathcal{P}_\rho \psi \dd x
        =
        \int_{\R^n}u\; \mathrm{div}\psi \dd x,
   \end{align*}
   hence by the fundamental theorem of calculus of variations, we deduce that $u$ is constant almost everywhere.
   Combining this with the fact that $Q_\rho \ast u=0$, we conclude that $u=0$ almost everywhere.
   Subsequently, $\mathcal{Q}_\rho$ is injective.\\
   Next, we prove that $\mathcal{Q}_\rho$ is surjective.
   Let $w\in W^{1,\infty}(\R^n)$ be arbitrary.
   Furthermore, let $V_\rho \in C^\infty(\R^n\setminus \{0\}; \R^n) \cap L^1_{\mathrm{loc}}$ be the kernel from~\cite[Theorem 5.2]{BellidoMorraCorralSchoenberger-NonLocGrad-FundamentalPoincareEmbeddings}.
   It satisfies $V_\rho \ast D^\rho \phi= \phi $ for all $\phi \in C^\infty_c(\R^n)$.
   Let $\chi\in C^\infty_c(\R^n)$ be a radial function such that $\chi(x)=1$ for all $\abs{x}\leq 1$.
   In pursuit of finding a preimage of $w$ under $\mathcal{Q}_\rho$, we define
    \begin{equation}\label{eq:PreimageOfQRho}
        v\ceq (\chi V_\rho)\ast \nabla w + \mathrm{div}((1-\chi)V_\rho)\ast w
    .
    \end{equation}
    The straightforward approach would be to define $\tilde{v}= V_\rho\ast \nabla w$, but since $\nabla w\in L^\infty$ and $V_\rho \in L^1_{\mathrm{loc}}$, the expression, in contrast to the one in~\cref{eq:PreimageOfQRho}, is not necessarily well-defined.
    As in~\cite[Lemma 4.5]{NonlocPerimetersVariations:ExtremalityAndDecomposability-CarioniDelFrandeIglesiasSchoenberger},
    we have that $\chi V_\rho$ and $\mathrm{div}((1-\chi)V_\rho)\in L^1$.
    Therefore, by Young's convolution inequality, we conclude that $v\in L^\infty$.
    For arbitrary $\psi \in C^\infty_c(\R^n; \R^n)$, one has that $\mathrm{div}(\psi)\in C^\infty_c(\R^n)$.
    Hence, we can use~\cref{eq:NonlocVC-divRhoVRho} and compute
    \begin{align*}
        \int_{\R^n}
            v\;
            \mathrm{div}^\rho \psi
        \dd x
        &=
        \int_{\R^n}
        \pars*{
            \chi V_\rho \ast \nabla w + \mathrm{div}((1-\chi)V_\rho)\ast w
        }
        \mathrm{div}^\rho \psi
        \dd x
        \\
        &=
        -
        \int_{\R^n}
        w \;
        \mathrm{div}
        \pars*{
            V_\rho \ast \mathrm{div}^\rho \psi
        }
        \dd x
        =
        -
        \int_{\R^n}
        w \;
        \mathrm{div}^\rho
        \pars*{
            V_\rho \ast \mathrm{div} \psi
        }
        \dd x
        \\
        &=
        -
        \int_{\R^n}
            w \; \mathrm{div}\psi
        \dd x
    .
    \end{align*}
    Therefore, $v\in H^{\rho,\infty}$.
    We are left with verifying that $\mathcal{Q}_\rho (-v)=w$.
    For a test function $\phi \in C^\infty_c(\R^n)$, one computes 
    \begin{align*}
        \int_{\R^n}
            \pars*{Q_\rho \ast v} \phi
        \dd x
        &=
        -
        \int_{\R^n}
           w \;
           \mathrm{div}
           \pars*{
            \chi V_\rho \ast Q_\rho \ast \phi
            +
            (1-\chi)V_\rho \ast Q_\rho \ast \phi
           } 
        \dd x
        \\
        &=
        -\int_{\R^n}
        w\;
        \mathrm{div}\pars*{
            V_\rho \ast \mathcal{Q}_\rho \phi
        }
        \dd x
        =
        -\int_{\R^n}
        w
        \pars*{V_\rho \ast \nabla \mathcal{Q}_\rho \phi}
        \dd x
        \\
        &=
        -\int_{\R^n}
            w\;
            \pars*{V_\rho \ast D^\rho \phi}
        \dd x
        =-
        \int_{\R^n}
            w \phi 
        \dd x.
    \end{align*}
    Hence, we conclude that $\mathcal{Q}_\rho (-v)=w$, implying surjectivity of $\mathcal{Q}_\rho$.
    By the Banach isomorphism theorem, we now conclude that $\mathcal{P}_\rho= (\mathcal{Q}_\rho)^{-1}: W^{1,\infty}(\R^n)\rightarrow H^{\rho,\infty}(\R^n)$ exists and is a continuous inverse of $\mathcal{Q}_\rho$.
    Using the first statement of the theorem, we compute that for $v\in W^{1,\infty}$
    \[
        \nabla v=
        \nabla \mathcal{Q}_\rho \mathcal{P}_\rho v
        =
        D^\rho \mathcal{P}_\rho v.
    \]
    This completes the proof.
\end{proof}
\begin{remark}\label{rmk:TranslationGradInfty}
    In the context of~\cref{prop:LusinNonlocalGradient}, the above theorem is only applied to function in $W^{1,\infty}$ with compact support.
    The existence of the operator $\mathcal{P}_\rho$ on $W^{1,\infty}_{c}(\R^n)$ follows from a straightforward computation.
    In fact, for $w\in W^{1,\infty}_c$, we define $v\ceq V_\rho \ast \nabla w$.
    Since $\nabla w$ is compactly supported and $V_\rho\in L^1_{\mathrm{loc}}$, one immediately concludes that $v\in L^\infty$.
    Furthermore, by testing with $\phi\in C^\infty_c$, one observes
    \begin{align*}
        \int_{\R^n} v\; \mathrm{div}^\rho \phi \dd x
        &=
        \int_{\R^n} \nabla w (V_\rho \ast \mathrm{div}^\rho \phi)\dd x
        =
        \int_{\R^n} \nabla w \phi \dd x.
    \end{align*}
    Hence, $D^\rho w \in L^\infty$ and therefore $w\in H^{\rho,\infty}$.
    The above theorem has been proven to extend the existing theory to the case $p=\infty$.
\end{remark}

In the situation of~\cref{thm:TranslationMechNonlocGrad-inftyCase}, we can prove a Lusin\NotSure{-}type approximation result, at the expense that we end up in a nonlocal Sobolev space instead of the classical Hölder space $C^{0,s}$.
\begin{proposition}\label{prop:LusinNonlocalGradient}
    Let $\rho\in \mathcal{A}$ with compact support, $\Omega\subset \R^n$ be an open set of finite measure and $f:\Omega\rightarrow \R^n$ a Borel function.
    For all $\eps>0$, there is a $v\in H^{\rho,\infty}(\R^n)$ and an open set $A$ such that
    \begin{align*}
        &\abs{A}\leq \eps \abs{\Omega},\\
        &f= D^\rho v \text{ in }\Omega \setminus A,\\
        &\norm{D^\rho v}_{L^p(\R^n)}\leq c(n) \eps^{1/p-1} \norm{f}_{L^p(\Omega)}
        \text{ for all }p\in \intervalcc{1,\infty}
        .
    \end{align*}
\end{proposition}
\begin{proof}
    We apply~\cite[Theorem 1]{LusinForGradients-Alberti} and \aref{rmk:Alberti-C1c}  
    with $\tilde{\eps}= \eps$ and obtain $u\in C^{1}_c(\Omega)$, such that $\nabla u =f$ in $\Omega \setminus A$.
    Extending $u$ by $0$ to $\R^n$, we observe that $u\in C_c^{0,1}(\R^n)\subset W_c^{1,\infty}(\R^n)$.
    In order to proceed, we have to use the local to nonlocal direction of the translation method, first proposed in~\cite[Proposition 2.6]{BellidoMorraCorralSchoenberger-NonLocGrad-FundamentalPoincareEmbeddings}.
    The general idea is, once again, to use the splitting
    \[
        \nabla u
        =
        D^\rho
        \mathcal{P}_\rho u
    ,
    \]
    where $\mathcal{P}_\rho : W^{1,p}\rightarrow H^{\rho,p}(\R^n)$, $1\leq p\leq\infty$, is the operator introduced in~\cref{eq:Intro-PRho}.
    Using~\cite[Lemma 2.8]{BellidoGarciaSaez-NonlocalElasticity}, we conclude that $\mathcal{P}_\rho$ is well-defined for $p\in (1,\infty)$.
    By \aref{thm:TranslationMechNonlocGrad-inftyCase}, this extends to $p=\infty$.
    Furthermore, the proof for $p=1$ essentially follows from~\cite[Lemma 4.5]{NonlocPerimetersVariations:ExtremalityAndDecomposability-CarioniDelFrandeIglesiasSchoenberger} where the authors proved the existence and continuity of $\mathcal{P}_\rho: BV(\R^n)\rightarrow BV^\rho(\R^n)$.
    For a definition of $BV^\rho$, see~\cref{def:BVRho} below.
    \\
    We define $v\ceq \mathcal{P}_\rho u$.
    Since $u\in W^{1,\infty}_c(\R^n) \subset W^{1,p}(\R^n)$ for all $p\in [1,\infty]$, we observe that $v\in H^{\rho,p}(\R^n)$ for all $p\in [1,\infty]$.
    Analogously to \aref{thm:LusinFractGradient}, the function $v$ satisfies $D^\rho v=f$ on $\Omega\setminus A$, with $\abs{A}\leq \eps \abs{\Omega}$.
    Ultimately, we verify the norm estimate.
    For $f\in L^p(\Omega), 1\leq p \leq \infty$, we compute
    \begin{align*}
        \norm{D^\rho v}_{L^p(\R^n)}
        &=
        \norm{D^\rho P_\rho u }_{L^p(\R^n)}
        =
        \norm{\nabla u}_{L^p(\R^n)}
        \\
        &=
        \norm{\nabla u}_{L^p(\Omega)}
        \leq c(n) \eps^{1/p-1}\norm{f}_{L^p(\Omega)}
        .
    \end{align*}
\end{proof}
It is natural to ask whether fractional gradients with finite horizons fall into this category.
The following remark sheds light on this situation.
\begin{remark}
    In particular, kernels of fractional gradients of order $s$, $\rho^s_\delta$, are admissible kernels with finite horizon.
    In this case, the function $v$ is in $H^{s,\infty}(\R^n)=W^{s,\infty}(\R^n)$ and therefore Hölder-continuous.
    \\
    Furthermore, if the gradient has horizon $\delta>0$, one can multiply the original function by a suitable cutoff $\chi$, with $\mathrm{spt}(\chi)\subset B_\delta(0)$.
    Hence, without loss of generality, one may assume that the function $v$ in \aref{prop:LusinNonlocalGradient} is supported in $\Omega_\delta$.
\end{remark}
Before extending our results to the $BV$ setting, we give a definition of said space.
\begin{definition}\label{def:BVRho}
    Let $\rho\in \mathcal{A}$ be an admissible kernel.
    A function $u\in L^1(\R^n)$ belongs to $BV^\rho(\R^n)$, if
    \[
    \mathrm{TV}_\rho(u)\ceq
        \sup\left\{
            \int_{\R^n} u \;\mathrm{div}^\rho \varphi \dd x
            \;
            \bigg\vert \;\varphi\in C^\infty_c(\R^n, \R^n),\; \norm{\varphi}_{L^\infty}\leq 1
        \right\}
        <\infty
    .
    \]
    If $\rho=\rho^s$, we abbreviate $BV^s\ceq BV^{\rho^s}$.
\end{definition}
We are now equipped to prove a version of~\cite[Theorem 3]{LusinForGradients-Alberti} for fractional gradients.
\begin{theorem}\label{thm:LusinFracGradBV}
    Let $\Omega \subset \R^n$ be open and $f\in L^1(\Omega)$.
    Then, there exist $u\in BV^s(\R^n)$ and a Borel function $g:\Omega\to \R^n$, such that 
    \begin{equation}
        \begin{split}
        &D^s u= f \dd \lambda^n + g \dd \Hmeasure^{n-1}
        \\
        &\int_{\Omega} \abs{g}\dd \Hmeasure^{n-1} \leq C(n) \norm{f}_{L^1(\Omega)}
        ,
        \end{split}
    \end{equation}
    where $\lambda^n$ is the $n$-dimensional Lebesgue measure and $\Hmeasure^{n-1}$ is the $n-1$-dimensional Hausdorff measure.
\end{theorem}
\begin{proof}
    In this proof, we will refer to measure-valued derivatives of functions $v$ as $Dv$.
    We apply~\cite[Theorem 3]{LusinForGradients-Alberti} in order to obtain $v\in BV(\R^n)$, such that
    \[
        Dv= f\dd \lambda^n + g\dd \Hmeasure^{n-1}, \int_{\Omega}\abs{g}\dd \Hmeasure^{n-1}\leq C(n)\norm{f}_{L^1(\Omega)}
    .
    \]
    Again, the main idea is invoking the translation principle, combined it with a suitable decomposition of the gradient and the divergence.
    Using~\cite[Proposition 2.1]{ComiStefani-Asymptotics-DistrApproachToFractionalSobolevSpaces}, we conclude that
    \begin{equation}\label{eq:FractDivSplit}
        \text{div}^s \phi
        =
        \text{div}\circ \mathcal{I}_{1-s} \phi
        =
        \mathcal{I}_{1-s}
        \circ
        \text{div}
        \phi
        \quad
        \text{and}
        \quad
        D^s \phi 
        =
        \nabla \circ \mathcal{I}_{1-s} \phi
        \quad
        \text{for }
        \phi \in \text{Lip}_c
    .
    \end{equation}
    We define $u\ceq (-\Delta)^{\frac{1-s}{2}}v$.
    Since $v\in BV\subset W^{\alpha,1}$ for all $\alpha\in (0,1)$, we observe that $u$ is well-defined.
    In order to show that $u\in BV^s$, we employ the symmetry of Laplacians and~\cref{eq:FractDivSplit}.
    For arbitrary $\phi \in C^\infty_c(\R^n; \R^n)$, we compute 
    \begin{align*}
        \int_{\R^n}
            u \text{div}^{s} \phi 
        \dd x
        &=
        \int_{\R^n}
            ((-\Delta)^{\frac{1-s}{2}}v) 
            \mathcal{I}_{1-s}
            \circ
            \text{div}
            \pars*{
                \phi
            }
        \dd x
        \\
        &=
        \int_{\R^n}
            v\;
            \pars*{
            (-\Delta)^{\frac{1-s}{2}} 
            \circ
                \mathcal{I}_{1-s}
                \circ
                    \text{div}
            }(\phi)
        \dd x
        =
        \int_{\R^n}
            v \;
            \text{div}\phi
        \dd x.
    \end{align*}
    Hence, $u\in BV^s$ and $D^s u= Dv= f \dd \lambda^n + g\dd \Hmeasure^{n-1}$.
    The claim now follows by~\cite[Theorem 3]{LusinForGradients-Alberti}.
\end{proof}
We now prove a version of~\cref{thm:LusinFracGradBV} for general nonlocal gradients, at the expense of assuming integrability and compact support.
\begin{proposition}\label{prop:LusinNonlocGradBV}
    Let $\Omega \subset \R^n$ be open, $\rho\in \mathcal{A}$ with compact support and $f\in L^1(\Omega)$.
    Then, there exist $u\in BV^\rho(\R^n)$ and a Borel function $g:\Omega\to\R^n$, such that
    \begin{equation}
        \begin{split}
        &D^\rho u= f \dd \lambda^n + g \dd \Hmeasure^{n-1}
        \\
        &\int_{\Omega} \abs{g}\dd \Hmeasure^{n-1} \leq C(n) \norm{f}_{L^1(\Omega)}
        \end{split}
        .
    \end{equation}
\end{proposition}
\begin{proof}
    Again, we use $D$ for measure-valued derivatives and $\nabla$ for ordinary derivatives.
    Analogously to before, we employ~\cite[Theorem 3]{LusinForGradients-Alberti} in order to obtain $v\in BV(\R^n)$, satisfying
    \[
        Dv= f\dd \lambda^n + g\dd \Hmeasure^{n-1}
        \quad
        \text{and} 
        \quad
        \int_{\Omega}\abs{g}\dd \Hmeasure^{n-1}\leq c(n)\norm{f}_{L^1(\Omega)}
    .
    \]
    Proceeding as in~\cref{thm:LusinFracGradBV}, we use the translation principle and a suitable decomposition of $\nabla$ and $\text{div}$.
    Note that, due to (H0) and $\rho$ having compact support, we observe that $\rho \in L^1$.
    Using~\cite[Proposition 2.6]{BellidoMorraCorralSchoenberger-NonLocGrad-FundamentalPoincareEmbeddings}, we conclude that for $\phi \in C^\infty_c(\R^n,\R^n), \psi\in C^\infty_c(\R^n,\R)$
    \[
        \text{div}^\rho \phi
        =
        \text{div} \mathcal{Q}_{\rho} \phi
        =
        \mathcal{Q}_{\rho} \text{div} \phi
        \quad
        \text{and}
        \quad
        D^\rho \psi= \mathcal{Q}_{\rho} \nabla \psi= \nabla \mathcal{Q}_\rho \psi
    .
    \]
    By~\cite[Lemma 4.3]{NonlocPerimetersVariations:ExtremalityAndDecomposability-CarioniDelFrandeIglesiasSchoenberger}, the operator $\mathcal{Q}_{\rho}$ extends to a mapping $BV^\rho\rightarrow BV$.
    Using~\cite[Lemma 4.5]{NonlocPerimetersVariations:ExtremalityAndDecomposability-CarioniDelFrandeIglesiasSchoenberger}, the operator $\mathcal{P}_\rho=\mathcal{Q}_{\rho}^{-1}$ exists and is well-defined.
    Let $u\ceq \mathcal{P}_\rho v\in BV^\rho$.
    For the sake of completeness, we compute $D^\rho u$.
    Hence, for arbitrary $\phi \in C^\infty_c$, we compute
    \begin{align*}
        \int_{\R^n}
            u \;\text{div}^{\rho} \phi 
        \dd x
        &=
        \int_{\R^n}
            \mathcal{P}_\rho v 
            \;
            \mathcal{Q}_{\rho}
            \text{div}\phi
        \dd x
        =
        \int_{\R^n}
            (P_\rho \ast v)\;
            (Q_\rho\ast \text{div} \phi)
        \dd x
        \\
        &=
        \int_{\R^n}
            v 
            (P_\rho \ast( Q_\rho \ast
            \text{div}\phi))
        \dd x
        =
        \int_{\R^n}
            v \;
            \text{div}\phi
        \dd x
        .
    \end{align*}
    Hence, $D^\rho u= D v= f \dd \lambda^n + g\dd \Hmeasure^{n-1}$, in the sense of measures.
    Again, the claim follows by~\cite[Theorem 3]{LusinForGradients-Alberti}.
\end{proof}
Note that, if $\omega_\delta \in \mathcal{C}_\delta$ is an admissible cut-off, then $\frac{1}{\abs{x}^{n+s-1}} \omega_\delta(x)\in \mathcal{A}\cap L^1$ and has compact support.
In particular,~\cref{prop:LusinNonlocGradBV} applies to the fractional gradient $D^s_\delta$ with horizon $\delta>0$.
Additionally, in the situation of~\cref{thm:LusinFracGradBV} and~\cref{prop:LusinNonlocGradBV}, one may extend $f$ by $0$ to $\R^n$ and therefore assume that $\Omega=\R^n$.

\begin{proposition}\label{prop:FractCurlZero}
    Let $\Omega\subset \R^n$ be an open set of finite measure and $f:\Omega\rightarrow \R^n$ be a Borel vector field.
    Let $(u_n)_{n\in \N}\subset H^{s,p}(\R^n)$ be a sequence of functions such that the sets $A_n\ceq \{ x\in \Omega: f \neq D^s u_n\}$ satisfy
    \[
        \abs{A_n}\rightarrow 0
        \quad
        \text{and}
        \quad
        \liminf_{n\to \infty} \abs{A_n}^{1/p-1}\norm{D^s u_n}_{L^p(\R^n)}=0
    .
    \]
    Then $\mathrm{curl}^s f=0$ in the sense of distributions.
\end{proposition}
\begin{proof}
    Once again, we use the decomposition of $D^s= \nabla \circ \mathcal{I}_{1-s}$.
    Using that $\mathcal{I}_{1-s}:H^{s,p}(\R^n)\rightarrow H^{1,p}(\R^n)$, we define $v_n \ceq \mathcal{I}_{1-s} u_n \in H^{1,p}(\R^n)$ and $h_n=v_n \vert_{\Omega} \in H^{1,p}(\Omega)$, which gives rise to
    \[
        B_n\ceq
        \{
            x\in \Omega:
            \nabla h_n \neq f
        \}
        =
        A_n
        \quad
        \text{and}
        \quad
        \norm{D^s u_n}_{L^p(\R^n)}
        =
        \norm{\nabla h_n}_{L^p(\R^n)}
        \geq
        \norm{\nabla h_n}_{L^p(\Omega)}
    .
    \]
    Therefore, by our assumption, we conclude that
    \[
        \abs{B_n}\rightarrow 0
        \quad
        \text{and}
        \quad
        \liminf_{n\to \infty} \abs{B_n}^{1/p -1}
        \norm{h_n}_{L^p(\Omega)}=0
        .
    \]
    Hence, we can apply Alberti's result for the local setting, cf.~\cite[Proposition 5]{LusinForGradients-Alberti}, in order to conclude that $\mathrm{curl}^s f=0$ in the sense of distributions.
\end{proof}
\cref{prop:FractCurlZero} generalizes to wider class of kernels, provided they have finite horizon $\delta>0$.
Furthermore, one can choose $(u_n)_{n\in \N} \subset H^{s,p}(\Omega_\delta)$ by multiplying with a suitable cut-off.

\section{$L^p$-Differentiability}\label{sec:ApproxDiffable}
Since Lusin-type approximation results are closely related to fine differentiability properties, we conclude by investigating $L^p$-differentiability and approximate differentiability.
Because $L^p$-differentiability is a stronger notion, we examine said property.
Additionally, $L^p$-differentiability is a global property whereas approximate differentiability is a local property.
Note that if $\Omega'\subset \Omega$, then $W^{\rho,p}(\Omega)\hookrightarrow W^{\rho,p}(\Omega')$ is continuous.
Therefore, the according statements concerning approximate differentiability immediately follow from the $L^p$-differentiability.
Nonetheless, we include the definition of approximate differentiability, because we use said property to derive fine properties of $W^{\rho,p}$-functions, cf.~\cref{rmk:FineProperties}. 
In order to do so, we follow the definition given in~\cite{EvansGariepyMeasureTheory}.
For a more detailed exposition, we refer to~\cite{FedererZiemerDistrDerivApproxDeriv} and the classical monograph~\cite{GMTFederer}.
Let us recall the notion of an \emph{approximate limit}.
A map $f:\R^n\to \R^m$ satisfies $\mathrm{ap}\lim_{y\to x}f(y)=a$, if there exists a measurable set $F\subset \R^n$ with density $1$ at $x$ such that
\[
    \lim_{F\ni y\to x}f(y)=a.
\]
The approximate limes superior and inferior are defined analogously.
\begin{definition}\label{def:ApproxDiffable}
    Let $k\in \N, \alpha \in \intervaloc{0,1}$ be given.
    A function $f:\R^n\rightarrow \R^m$ is approximately differentiable of class $k$ at $x$ if
    there exists a polynomial $P_x$ with degree at most $k$, such that
    \begin{equation}
        \mathrm{ap}\lim_{y\to x}
            \frac{\abs{f(y)- P_x(y)}}{\abs{y-x}^k}=0
        .
    \end{equation}
    We say that $f$ is approximately differentiable of class $(k,\alpha)$ if
    \begin{equation}
        \mathrm{ap}\limsup_{y\to x}
            \frac{\abs{f(y)- P_x(y)}}{\abs{y-x}^{k+\alpha}}<\infty
        .
    \end{equation}
\end{definition}
A slightly stronger generalization of differentiability was introduced in~\cite{zbMATH03312891,zbMATH03162202}, the \emph{$L^p$-differentiability}.
\begin{definition}
    Let $\Omega\subset \R^n$ be open, $f\in L^p(\Omega)$, $p\in \intervalcc{1,\infty}$ $k\in \N$ and $\alpha \in \intervaloc{0,1}$ be given.
    We say that f is \emph{$L^p$-differentiable of class $k$ at $x\in \Omega$}, if there exists a polynomial $P_x$ with degree at most $k$ such that
    \[  
        \lim_{\eps\to 0}
            \frac{1}{\abs{B_\eps(x)\cap \Omega}}
            \int_{B_\eps(x)\cap \Omega}\abs*{\frac{f(y)-P_x(y-x)}{\eps^k}}^p 
            \dd y
        =0.
    \]
    Subsequently, we say that $f$ is \emph{$L^p$-differentiable of class $k$ in $\Omega$} if $f$ is $L^p$-differentiable of class $k$ at almost all $x\in \Omega$ and the function $x\mapsto \norm{P_x}$ is in $L^p(\Omega)$.
    Here, the norm of a polynomial of degree $P(x)=\sum_{l=0}^k a_l x^l$ is given by 
    \[
        \norm{P}\ceq
        \sum_{l=0}^k
        \norm{a_l}_{\mathrm{op}}
        .
    \]
    The function $f$ is  \emph{$L^p$-differentiable of class $(k,\alpha)$ at $x\in \Omega$} if there is a polynomial $P_x$ with $\deg(P_x)\leq k$ and $h_x\geq 0$ such that
    \begin{equation}\label{eq:ApproxLpDiffFractOrder}
        \limsup_{\eps\to 0}
            \frac{1}{\abs{B_\eps(x)\cap \Omega}}
            \int_{B_\eps(x)\cap \Omega}\abs*{\frac{f(y)-P_x(y-x)}{\eps^{k+\alpha}}}^p 
            \dd y
        \leq
        h_x^p.
    \end{equation}
    Again, we refer to $f$ as \emph{$L^p$-differentiable of class $(k,\alpha)$ in $\Omega$} if~\cref{eq:ApproxLpDiffFractOrder} holds true at almost every $x\in \Omega$ and the functions $x\mapsto \norm{P_x}, x\mapsto h_x$ are in $L^p(\Omega)$.
\end{definition}
Let us draw a parallel to Lebesgue point and points of approximate continuity.
Classical Lebesgue points can be interpreted as points where the function is $L^1$-differentiable of class $0$, while points of approximate continuity can be interpreted as points, where the function is approximately differentiable of class $0$.
It is known that every Lebesgue point is a point of approximate continuity, but the reverse statement is false.
Analogously, every point where a function is $L^1$-differentiable of class $k$ is a point where the function is approximately differentiable of class $k$ but not vice versa.
\begin{remark}
    Note that there are also intermediate notions, such as approximate $L^p$ differentiability.
    Furthermore, we would like to point out that the notion of $L^1$-differentiability is frequently referred to as approximate differentiability.
    For a detailed exhibition and its generalization to sets and distributions, we refer to~\cite{zbMATH07050120,zbMATH07403055}.   
\end{remark}
For $p\in \intervalco{1,\infty}, s\in \intervaloo{0,1}$, we know that a function $u\in W^{s,p}(\R^n)$ is $L^p$-differentiability of class $(0,s)$, see~\cite[Proposition 4.3]{HashashFinePropOfBesovSpaces}.
Since $H^{s,p}(\R^n)=W^{s,p}(\R^n)$, $L^p$-differentiability holds true for $H^{s,p}(\R^n)$ if $p\in \intervaloo{1,\infty}$.
Approximate and $L^p$-differentiability is still open in the case $p=1$.
The following proposition is an extension of~\cite{HashashFinePropOfBesovSpaces} to more general kernels.
\begin{proposition}
    Let $\Omega\subset \R^n$ be an open set and $\rho\in \mathcal{A}$ an admissible kernel, such that $\inf_{\intervaloo{0,\eta}}\overline{\rho} >0$ for a radius $\eta>0$.
    Assume that 
    \[
         r\mapsto r^{n+\sigma-1}\overline{\rho} \text{ is nonincreasing on }(0,\eta)
    \]
    for some $\sigma \in (0,1)$.
    Then any $u\in W^{\rho,p}(\Omega)$, $1\leq p\leq \infty$, is $L^p$-differentiable of order $(0,\sigma)$ in $\Omega$.
\end{proposition}
\begin{proof}
    Let $u,\sigma, \rho, \eta$ be as above and $x\in \Omega$ be arbitrary.
    Firstly, we investigate the case $p<\infty$.
    We define 
    \[  
        h(x)\ceq 
        \pars*{\int_{\Omega}
        \pars*{\abs{u(y)-u(x)} \rho(y-x) \abs{y-x}^{n-1}}^p
        \frac{\dd y}{\abs{y-x}^n}
        }^{1/p}
    .
    \]
    Since $u\in W^{\rho,p}$, we observe that $h\in L^p$.
    Furthermore, by our choice of $\sigma$, we infer
    \[  
        \sup_{r\in (0,\eta)}
        \frac{1}{r^{n+\sigma -1}\overline{\rho(r)}}
        \leq
        \frac{1}{\eta^{n+\sigma -1}\overline{\rho(\eta)}}
        \leq 
        \frac{1}{\eta^{n+\sigma-1}}(\inf_{\intervaloo{0,\eta}}\overline{\rho})^{-1}<\infty
        .
    \]
    Hence, for $r\in (0,\eta)$ such that $B_r(x)\subset \Omega$ and $\abs{u(x)},\abs{h(x)}<\infty$, one computes
    \begin{align*}
        &\frac{1}{r^n}
            \int_{B_r(x)}
                \abs{u(y)-u(x)}^p
            \dd y
        \leq
            \int_{B_r(x)}
                \abs{u(y)-u(x)}^p
            \frac{\dd y}{\abs{y-x}^n}
        \\
        &=
        \int_{B_r(x)}
            \pars*{
            \abs{u(y)-u(x)}
            \frac{\rho(y-x)\abs{y-x}^{n-1+\sigma}}{\rho(y-x)\abs{y-x}^{n-1+\sigma}}
            }^p
        \frac{\dd y}{\abs{y-x}^n}
        \\
        &\leq 
        r^{\sigma p}
        \int_{B_r(x)}
            \pars*{
            \abs{u(y)-u(x)}
            \frac{\rho(y-x)\abs{y-x}^{n-1}}{\rho(y-x)\abs{y-x}^{n-1+\sigma}}
            }^p
        \frac{\dd y}{\abs{y-x}^n}
        \\
        &\leq
        r^{\sigma p}
        (\eta^{n-1+\sigma}\inf_{(0,\eta)} \overline{\rho})^{-p}
        \int_{B_r(x)}
            \pars*{
            \abs{u(y)-u(x)}
            \rho(y-x)\abs{y-x}^{n-1}
            }^p
        \frac{\dd y}{\abs{y-x}^n}
        \\
        &\leq
        r^{\sigma p}
        (\eta^{n-1+\sigma}\inf_{(0,\eta)} \overline{\rho})^{-p}
        h(x)^p
        <\infty
        .
    \end{align*}
    Since $h$ is in $L^p$, we know that $h$ is finite at almost every $x\in \Omega$. 
    Hence, for such $x$ we conclude
    \[
        \limsup_{r\searrow 0}
        \frac{1}{r^{\sigma p} \abs{B_r(x)\cap \Omega}}
        \int_{B_r(x)\cap \Omega}
            \abs{u(y)-u(x)}^p
        \dd y 
        \leq
        C(\rho, \eta, \sigma, p, n)\abs{h(x)}^p
        <\infty
    ,
    \]
    showing that $u$ is $L^p$-differentiable of class $(0,\sigma)$ at $x$ with $P_x(y)= u(x)$.
    Due to $h$ being in $L^p$, we additionally observe that
    \[
    \int_{\Omega} 
        \pars*{ 
            \limsup_{r\searrow 0}
            \frac{1}{r^{\sigma p}\abs{\Omega \cap B_r(x)}}
            \int_{\Omega\cap B_r(x)}
            \abs{u(y)-u(x)}^p
            \dd y
        }
        \dd x
    \leq 
    C(\rho, \eta, \sigma, p, n)
    \norm{h}_{L^p}^p<\infty
    .
    \]
    Therefore, $u$ is $L^p$-differentiable of class $(0,\sigma)$ in $\Omega$.\\
    We now consider $p=\infty$.
    Hence, we have to show
    \begin{align*}
        \limsup_{r\to 0}\norm*{\frac{u(y)-u(x)}{r^\sigma}}_{L^\infty(B_r(x))}
        \leq h_x 
        \qquad
        \text{and}
        \qquad
        x\mapsto h_x \in L^\infty(\Omega).
    \end{align*}
    Once again, we define
    \[
        h_x(y)\ceq
        \abs{u(y)-u(x)}\rho(y-x)\abs{y-x}^{n-1}
    .
    \]
    By assumption, we have that $\pars*{(x,y)\mapsto h_x(y)} \in L^\infty(\Omega\times\Omega, \frac{\dd y \dd x}{\abs{y-x}^{n}})$.
    Defining the measure $\mu$ by $\dd \mu= \abs{y-x}^{-n} \dd y \dd x$, we observe that $\mu$ is absolutely continuous with respect to $\lambda^n\otimes \lambda^n$.
    Therefore, if 
    \[
        \pars*{(x,y)\mapsto h_x(y)} \in L^\infty(\Omega\times\Omega, \frac{\dd y \dd x}{\abs{y-x}^{n}})
        \text{, then } 
        \pars{(x,y)\mapsto h_x(y)}\in L^\infty(\Omega\times \Omega, \dd x \dd y).
    \]
    For $0<r<\eta$ and $x\in \Omega$, $y\in B_r(x)$ such that $\abs{h_x(y)}<\infty$, we observe that
    \begin{align*}
        &\frac{\abs{u(y)-u(x)}}{r^\sigma}
        \leq
        \pars*{
            \abs{u(y)-u(x)}\rho(y-x)\abs{y-x}^{n-1}
        }   
        \pars*{
            \rho(y-x)\abs{y-x}^{\sigma+n-1}
        }^{-1}
        \\
        &\leq
        h_x(y)
        \pars*{
            \eta^{n+\sigma-1}
            \overline{\rho}(\eta)
        }^{-1}
        \leq
        \norm{(x,y)\mapsto h_x(y)}_{L^\infty(\Omega\times\Omega, \frac{\dd y \dd x}{\abs{y-x}^{n}})}
        \pars*{
            \eta^{n+\sigma-1}
            \inf_{\intervaloo{0,\eta}}\overline{\rho}
        }^{-1}
        <\infty
        .
    \end{align*}
    This proves the claim.
\end{proof}
We end this section by proving an embedding result on domains of finite measure.
\begin{proposition}\label{prop:EmbeddingNonlocSoboSlobo}
    Let $\Omega\subset \R^n$ be a domain with finite measure and $\rho\in \mathcal{A}$ an admissible kernel, such that $\inf_{\intervaloo{0,\eta}}\overline{\rho} >0$ for a radius $\eta>0$.
    Let $\sigma \in (0,1)$ be given, such that 
    \[
         r\mapsto r^{n+\sigma-1}\overline{\rho} \text{ is almost nonincreasing on }(0,\eta).
    \]
    Then, the embedding $W^{\rho,p}(\Omega)\hookrightarrow W^{\sigma,p}(\Omega)$ is continuous.
    For $p=\infty$, the assumption $\abs{\Omega}<\infty$ can be dropped.
\end{proposition}
\begin{proof}
    Firstly, we consider $p=\infty$.
    For $\abs{y-x}\leq \eta$, one computes
    \begin{align*}
        \frac{\abs{u(y)-u(x)}}{\abs{y-x}^\sigma}
        &=
        \pars*{
        \abs{u(y)-u(x)}
        \rho(y-x)\abs{y-x}^{n-1}
        }
        \pars*{
            \rho(y-x)
            \abs{y-x}^{\sigma+n-1}
        }^{-1}
        \\
        &\leq
        \seminorm{u}_{W^{\rho,\infty}}\eta^{1-n-\sigma}
        \pars*{\inf_{\intervaloo{0,\eta}}\overline{\rho}}^{-1}
    .
    \end{align*}
    Furthermore, for $\abs{y-x}\geq \eta$, we have that
    \[
        \frac{\abs{u(y)-u(x)}}{\abs{y-x}^\sigma}
        \leq
        2\norm{u}_{L^\infty}
        \eta^{-\sigma}
    .
    \]
    This implies that $W^{\rho,\infty}(\Omega)\hookrightarrow W^{\sigma, \infty}(\Omega)$ is continuous.
    Since both estimates are independent of $\abs{\Omega}$, the embedding is also continuous if $\abs{\Omega}=\infty$.
    Next, we consider $p\in \intervalco{1,\infty}$.
    Once again, for $0<\abs{y-x}<\eta$, we have that 
    \begin{align*}
        \frac{\abs{u(y)-u(x)}}{\abs{y-x}^\sigma}
        \leq
        \pars*{
        \abs{u(y)-u(x)}
        \rho(y-x)\abs{y-x}^{n-1}
        }
        \eta^{1-n-\sigma}
        \pars*{\inf_{\intervaloo{0,\eta}}\overline{\rho}}^{-1}
    \end{align*}
    and for $\eta \leq \abs{y-x}$
    \[
        \frac{\abs{u(y)-u(x)}}{\abs{y-x}^\sigma}
        \leq
        \pars*{
        \abs{u(y)}
        +
        \abs{u(x)}}
        \eta^{-\sigma}
    .
    \]
    Hence,
    \[
    \seminorm{u}_{W^{\sigma,p}(\Omega)}^p
    \leq
    \pars*{
    \seminorm{u}_{W^{\rho,p}(\Omega)}^p\eta^{p(1-n-\sigma)}
    (\inf_{\intervaloo{0,\eta}}\overline{\rho})^{-p}
    +2^p\eta^{-p\sigma-n}\norm{u}_{L^p}^p \abs{\Omega}
    }.
    \]
    The assumption on the finiteness of $\abs{\Omega}$ stems from the fact, that 
    \[
        (x,y)\mapsto
        \pars*{
            \abs{u(y)}+\abs{u(x)}
        }^p
        \abs{y-x}^{-n}
    \]
    is not necessarily integrable on $\{ (x,y)\in \Omega^2: \abs{y-x}\geq \eta\}$.
\end{proof}
\begin{remark}\label{rmk:FineProperties}
    In the situation of~\cref{prop:EmbeddingNonlocSoboSlobo}, we observe the following:\\
    Since $L^p$-differentiability is a stronger notion than approximate differentiability, we conclude that every $u\in W^{\rho,p}$ is approximately differentiable of order $(0,\sigma)$.
    Using the characterizations of approximate differentiability from~\cite{GMTFederer} and~\cite{FedererZiemerDistrDerivApproxDeriv}, we conclude that for every $\eps>0$, there exists a closed set $C\subset \Omega$ and $v\in C^{0,\sigma}(\Omega)$ such that
    \[
        u=v \text{ on } C
        \qquad
        \text{and}
        \qquad
        \abs{\Omega\setminus C}<\eps.
    \]
    Therefore, functions $u\in W^{\rho,p}$ coincides with a $C^{0,\sigma}$-function on an arbitrary large set.
\end{remark}
\noindent
\textbf{Acknowledgement}\\
I am deeply indebted to Carolin Kreisbeck for her encouragement and the discussions we had concerning this work. 
I am also deeply grateful to Ulrich Menne and Akram Sharif for introducing me to this class of problems and for their patient encouragement as I developed my often unstructured ideas. 
Finally, I would like to thank Giacomo Del Nin and Philipp Reiter for the many insightful and stimulating discussions that greatly improved my understanding of this problem and this article.
Funding support from the European Union and the Free State of Saxony (ESF) is gratefully acknowledged.

\let\oldthebibliography\thebibliography
\let\endoldthebibliography\endthebibliography
\renewenvironment{thebibliography}[1]{
  \begin{oldthebibliography}{#1}
	\footnotesize
    \setlength{\itemsep}{0em}
    \setlength{\parskip}{0em}
}
{
  \end{oldthebibliography}
}

\bibliographystyle{abbrvhref}
\bibliography{bibliothek}

\end{document}